\documentclass[a4paper,12pt]{article}

\usepackage{color}
\usepackage{tikz}

\advance \textwidth -60truemm\relax

\advance\oddsidemargin -1truein\relax
\evensidemargin=\oddsidemargin

\advance\textheight -60truemm\relax
\advance\topmargin -1truein\relax
\unitlength\textwidth
\divide\unitlength by 200\relax

\usepackage{amsmath,amssymb}
\usepackage{bm}

\bmdefine{\sss}{s}
\bmdefine{\vvv}{v}

\DeclareMathAlphabet{\mathscr}{U}{rsfs}{m}{n}

\newcommand{\msPPP}{\mathscr{P}}

\newcommand{\msXXX}{\mathscr{X}}
\newcommand{\msKKK}{\mathscr{K}}

\newcommand{\NNN}{\mathbb{N}}
\newcommand{\ZZZ}{\mathbb{Z}}
\newcommand{\QQQ}{\mathbb{Q}}
\newcommand{\RRR}{\mathbb{R}}
\newcommand{\KKK}{\mathbb{K}}

\newcommand{\pppp}{\mathfrak{p}}

\newcommand{\UUUUU}{{\mathcal U}}

\newcommand{\tUUUUU}{{{t}\mathcal U}}
\newcommand{\qUUUUU}{{{q}\mathcal U}}

\newcommand{\UUUUUn}{{\mathcal U}^{(n)}}

\newcommand{\grade}{\mathrm{grade}}
\newcommand{\define}{\mathrel{:=}}

\newcommand{\gor}{Gorenstein}
\newcommand{\cm}{Cohen-Macaulay}

\newcommand{\relint}{{\rm{relint}}}

\newcommand{\trace}{{\mathrm{tr}}}

\newcommand{\Div}{{\mathrm{Div}}}
\renewcommand{\hom}{{\mathrm{Hom}}}
\newcommand{\spec}{{\mathrm{Spec}}}

\newcommand{\aff}{\mathrm{aff}}
\newcommand{\conv}{\mathrm{conv}}
\newcommand{\stab}{\mathrm{STAB}}
\newcommand{\hstab}{\mathrm{HSTAB}}
\newcommand{\tstab}{\mathrm{TSTAB}}
\newcommand{\qstab}{\mathrm{QSTAB}}

\newcommand{\xip}{\xi^{+}}

\newcommand{\ekp}{E_\KKK[\msPPP]}

\newcommand{\eksg}{{E_\KKK[\stab(G)]}}
\newcommand{\ekhsg}{{E_\KKK[\hstab(G)]}}
\newcommand{\ektsg}{{E_\KKK[\tstab(G)]}}
\newcommand{\ekqsg}{{E_\KKK[\qstab(G)]}}

\newtheorem{thm}{Theorem}[section]
\newtheorem{fact}[thm]{Fact}
\newtheorem{example}[thm]{Example}
\newtheorem{lemma}[thm]{Lemma}
\newtheorem{cor}[thm]{Corollary}
\newtheorem{definition}[thm]{Definition}
\newtheorem{prop}[thm]{Proposition}
\newtheorem{remark}[thm]{Remark}

\newcommand{\bigzerou}{\smash{\lower1.7ex\hbox{\bg 0}}}

\newcommand{\bigastu}{\smash{\lower1.7ex\hbox{\bg *}}}

\newcommand{\refeq}[1]{(\ref{#1})}
\numberwithin{equation}{section}

\newcommand{\mylabel}[1]{{\label{#1}\tt [#1]}}
\let\mylabel=\label

\title{%
On the \gor\ property of the Ehrhart ring of the stable set polytope
of an h-perfect graph%
}

\author{Mitsuhiro MIYAZAKI%
}
\date{\normalsize
Department of Mathematics, Kyoto University of Education,\\
1 Fujinomori, Fukakusa, Fushimi-ku, Kyoto, 612-8522, Japan\\
g53448@kyokyo-u.ac.jp%
}

\begin{document}

\maketitle

\sloppy

\begin{abstract}
In this paper,
we give a criterion of the \gor\  property of
the Ehrhart ring of the stable set polytope of
an h-perfect graph:
the Ehrhart ring of the stable set polytope of an h-perfect graph $G$ is
\gor\ if and only if
(1) 
sizes of maximal cliques are constant (say $n$) and
(2)
(a)
$n=1$, 
(b)
$n=2$ and there is no odd cycle without chord and length at least 7 or
(c)
$n\geq 3$ and there is no odd cycle without chord and length at least 5.
\\
Key Words: \gor\ ring, h-perfect graph, Ehrhart ring, stable set polytope
\\
MSC: 13H10, 52B20, 13C14
\end{abstract}

\section{Introduction}

Recently, Hibi and Tsuchiya \cite{ht} showed that the Ehrhart ring of the
stable set polytope of an odd cycle graph is \gor\ if and only if the length
of the cycle is less than or equal to 5.
They used the fact that a cycle graph is t-perfect.
On the other hand, Ohsugi and Hibi \cite{oh} showed that the Ehrhart ring
of the stable set polytope of a perfect graph is \gor\ if and only if 
all maximal cliques have the same size.

Meanwhile, Sbihi and Uhry \cite{su} introduced the notion of h-perfect graphs
as a common generalization of perfect and t-perfect graphs:
a graph is h-perfect if its stable set polytope is defined by the constraints
nonnegativity of vertices and the constraints corresponding to cliques and
odd cycles.

In this paper, we characterize \gor\ property of the Ehrhart ring of the
stable set polytope of an h-perfect graph
by using the trace of the canonical module of a \cm\ ring.
This result is a generalization of both results of Ohsugi-Hibi and Hibi-Tsuchiya.

\section{Preliminaries}

\mylabel{sec:pre}

In this section, we establish notation and terminology.
In this paper, all rings and algebras are assumed to 
be commutative 
with an identity element 
unless stated otherwise.
Further, all graphs are finite simple graphs without loop.
We denote the set of nonnegative integers, 
the set of integers, 
the set of rational numbers and 
the set of real numbers
by $\NNN$, $\ZZZ$, $\QQQ$ and $\RRR$ respectively.
For a set $X$, we denote by $\#X$ the cardinality of $X$.
For sets $X$ and $Y$, we define $X\setminus Y\define\{x\in X\mid x\not\in Y\}$.
For nonempty sets $X$ and $Y$, we denote the set of maps from $X$ to $Y$ by $Y^X$.
If $X$ is a finite set, we identify $\RRR^X$ with the Euclidean space 
$\RRR^{\#X}$.
For $f$, $f_1$, $f_2\in\RRR^X$ and $a\in \RRR$,
we define
maps $f_1\pm f_2$ and $af$ 
by
$(f_1\pm f_2)(x)=f_1(x)\pm f_2(x)$ and
$(af)(x)=a(f(x))$
for $x\in X$.
We denote the zero map $X\ni x\mapsto 0\in \RRR$ by $0$.
%For $f\in\RRR^X$, we set $\supp f\define\{x\in X\mid\xi(x)\neq 0\}$.
Let $A$ be a subset of $X$.
We define the characteristic function $\chi_A\in\RRR^X$ by
$\chi_A(x)=1$ for $x\in A$ and $\chi_A(x)=0$ for $x\in X\setminus A$.
For a nonempty subset $\msXXX$ of $\RRR^X$, we denote by $\conv\msXXX$
(resp.\ $\aff\msXXX$)
the convex hull (resp.\ affine span) of $\msXXX$.

\begin{definition}
\mylabel{def:xi+}
\rm
Let $X$ be a finite set and $\xi\in\RRR^X$.
For $B\subset X$, we set $\xip(B)\define\sum_{b\in B}\xi(b)$.
We define the empty sum to be 0, i.e., if $B=\emptyset$, then $\xi^+(B)=0$.
\end{definition}

A stable set of a graph $G=(V,E)$ is a subset $S$ of $V$ with no two
elements of $S$ are adjacent.
We treat the empty set as a stable set.

\begin{definition}
\rm
The stable set polytope $\stab(G)$ of a graph $G=(V,E)$ is
$$
\conv\{\chi_S\in\RRR^V\mid \mbox{$S$ is a stable set of $G$.}\}
$$
\end{definition}
It is clear that for $f\in\stab(G)$,

\begin{enumerate}
\item
\mylabel{item:nonneg}
$f(x)\geq 0$ for any $x\in V$.
\item
\mylabel{item:clique}
$f^+(K)\leq 1$ for any clique $K$ in $G$.
\item
\mylabel{item:cycle}
$f^+(C)\leq\frac{\#C-1}{2}$ for any odd cycle $C$.
\end{enumerate}

\begin{definition}
\rm
We set
$$
\hstab(G)\define\{f\in \RRR^V\mid\mbox{$f$ satisfies \ref{item:nonneg}, 
\ref{item:clique} and \ref{item:cycle} above}\}.
$$
If $\hstab(G)=\stab(G)$, 
then we say that $G$ is an h-perfect graph.
\end{definition}
It is immediately seen that $\stab(G)\subset\hstab(G)$
for a general graph $G$.
Note that $\hstab(G)$ is a convex polytope since it is bounded by
\ref{item:nonneg} and 
\ref{item:clique}.

We fix notation about Ehrhart rings.
Let $\KKK$ be a field, $X$ a finite set  
and $\msPPP$ a rational convex polytope in $\RRR^X$, i.e., 
a convex polytope whose vertices are contained in $\QQQ^X$.
Let $-\infty$ be a new element
with $-\infty\not\in X$ and
set $X^-\define X\cup\{-\infty\}$.
Also let $\{T_x\}_{x\in X^-}$ be
a family of indeterminates indexed by $X^-$.
For $f\in\ZZZ^{X^-}$, 
we denote the Laurent monomial 
$\prod_{x\in X^-}T_x^{f(x)}$ by $T^f$.
We set $\deg T_x=0$ for $x\in X$ and $\deg T_{-\infty}=1$.
Then the Ehrhart ring of $\msPPP$ over a field $\KKK$ is the $\NNN$-graded subring
$$
\KKK[T^f\mid f\in \ZZZ^{X^-}, f(-\infty)>0, \frac{1}{f(-\infty)}f|_X\in\msPPP]
$$
of the Laurent polynomial ring $\KKK[T_x^{\pm1}\mid x\in X][T_{-\infty}]$,
where $f|_X$ is the restriction of $f$ to $X$.
We denote the Ehrhart ring of $\msPPP$ over $\KKK$ by $\ekp$.

It is known that $\ekp$ is Noetherian.
Therefore normal and \cm\ by the criterion of Hochster \cite{hoc}.
Further, 
by the description of the canonical module of a normal affine semigroup ring
by Stanley \cite[p.\ 82]{sta2}, we see 
the following.

\begin{lemma}
\mylabel{lem:sta desc}
The ideal 
$$\bigoplus_{f\in\ZZZ^{X^-}, f(-\infty)>0, \frac{1}{f(-\infty)}f|_X\in\relint\msPPP}
\KKK T^f
$$
of $\ekp$ is the canonical module of $\ekp$,
where $\relint\msPPP$ denotes the interior of $\msPPP$ in the topological space
$\aff\msPPP$.
\end{lemma}
We denote the ideal of the above lemma by $\omega_{\ekp}$ and call the canonical ideal of $\ekp$.

%It is also known that the dimension (Krull dimension) of $\ekp$ is equal to $\dim \msPPP+1$.

Let $R$ be a Noetherian normal domain.
Then the set of divisorial ideals $\Div(R)$ form a group by the operation
$I\cdot J\define R:_{Q(R)}(R:_{Q(R)} IJ)$ for $I$, $J\in\Div(R)$,
where $Q(R)$ is the quotient field of $R$.
See e.g., \cite[Chapter I]{fos} for details.
We denote the $n$-th power of $I\in\Div(R)$ in this group by $I^{(n)}$.
Note that if $R$ is a \cm\ local or graded ring over a field with canonical
module $\omega$, then $\omega$ is isomorphic to 
a divisorial ideal.
See e.g., \cite[Chapter 3]{bh} for details.

%%%%%%%%%%%%%%%%%%%%%%%%%%%%%%%%%%%%%%%%%%%%%%

\section{The trace of the canonical module and \gor\ property}

\mylabel{sec:gor}

In this section, we give a criterion of \gor\ property of the Ehrhart ring
of $\hstab(G)$ of a graph $G$.
As a consequence, we give a criterion of \gor\ 
property of the Ehrhart ring of the stable set polytope of an h-perfect graph.
First we recall the following.

\begin{definition}
\rm
Let $R$ be a ring and $M$ an $R$-module.
We set 
$$
\trace(M)\define\sum_{\varphi\in\hom(M,R)}\varphi(M)
$$
and call $\trace(M)$ the trace of $M$.
\end{definition}

Next we recall the following basic fact on the trace of an ideal.

\begin{fact}[{\cite[Lemma 1.1]{hhs}}]
\mylabel{fact:hhs1.1}
Let $R$ be a Noetherian ring and $I$ an ideal of $R$ with $\grade I>0$.
Also let $Q(R)$ be the total quotient ring of fractions of $R$
and set $I^{-1}\define\{x\in Q(R)\mid xI\subset R\}$.
Then
$$
\trace(I)=I^{-1}I.
$$
\end{fact}
Note that if $R$ is a domain and $I$ is a divisorial ideal, then $I^{-1}=I^{(-1)}$.
Moreover, we recall the following.

\begin{fact}[{\cite[Lemma 2.1]{hhs}}]
\mylabel{fact:hhs2.1}
Let $R$ be a \cm\ local or graded ring over a field with canonical
module $\omega_R$.
Then for $\pppp\in\spec(R)$,
$$
R_\pppp\mbox{ is \gor}\iff
\pppp\not\supset\trace(\omega_R).
$$
In particular, $R$ is \gor\ if and only if $\trace(\omega_R)\ni 1$.
\end{fact}

In the rest of this paper, we fix a graph $G=(V,E)$.
First we note the following.

\begin{lemma}
\mylabel{lem:ess cond}
Let $f\in\RRR^V$.
Then $f\in \hstab(G)$ if and only if
\begin{enumerate}
\item
\mylabel{item:ess nonneg}
$f(x)\geq 0$ for any $x\in V$,
\item
\mylabel{item:max clique}
$f^+(K)\leq 1$ for any maximal clique $K$ of $G$ and
\item
\mylabel{item:odd cycle}
$f^+(C)\leq\frac{\#C-1}{2}$ for any odd cycle $C$ without chord
with length at least 5.
\end{enumerate}
\end{lemma}

In order to make statements simple, we make the following.

\begin{definition}
\rm
\mylabel{def:un}
Let $n\in\ZZZ$.
We set
$\UUUUUn\define\{\mu\in\ZZZ^{V^-}\mid
\mu(z)\geq n$ for any $z\in V$,
$\mu^+(K)\leq \mu(-\infty)-n$ for any maximal clique $K$ of $G$ and
$\mu^+(C)\leq\mu(-\infty)\frac{\#C-1}{2}-n$ for any odd cycle $C$ without
chord and length at least $5\}$.
\end{definition}
Note that if $\mu\in \UUUUU^{(n)}$ and $\mu'\in\UUUUU^{(m)}$, then
$\mu+\mu'\in\UUUUU^{(n+m)}$.
By Lemmas \ref{lem:sta desc}, \ref{lem:ess cond} and Definition \ref{def:un}, we see the
following.

\begin{cor}
\mylabel{cor:in e and omega}
Let $\mu\in\ZZZ^{V^-}$.
Then $T^\mu\in\ekhsg$ if and only if $\mu\in\UUUUU^{(0)}$.
Moreover, $T^\mu\in\omega_{\ekhsg}$ if and only if
$\mu\in\UUUUU^{(1)}$.
\end{cor}

Next we show the following.

\begin{prop}
\mylabel{prop:symb power}
Let $n\in\ZZZ$ and $\mu\in\ZZZ^{V^-}$
and
set $\omega\define\omega_\ekhsg$.
Then $T^\mu\in\omega^{(n)}$ if and only if $\mu\in\UUUUU^{(n)}$.
\end{prop}
\begin{proof}
The cases where $n=0$ and $n=1$ are stated in 
Corollary \ref{cor:in e and omega}.
Next, we consider the case where $n<0$.
Set $m=-n$.
First we show the ``if'' part of this case.
Let $\mu_1$, \ldots, $\mu_m$ be arbitrary elements of $\UUUUU^{(1)}$.
Then, since $\mu\in\UUUUU^{(n)}$, we see that
$\mu+\mu_1+\cdots+\mu_m\in\UUUUU^{(0)}$.
Therefore, by Corollary \ref{cor:in e and omega}, we see that
$T^\mu T^{\mu_1}\cdots T^{\mu_m}\in\ekhsg$.
Since $T^{\mu_1}$, \ldots, $T^{\mu_m}$ are arbitrary Laurent monomials in
$\omega$ by Corollary \ref{cor:in e and omega}, we see that
$$
T^\mu\in \ekhsg:\omega^m=\omega^{(n)}.
$$

Next, we show the ``only if'' part of this case.
Let $N$ be a huge integer ($N>\#V$).
First let $\mu_1$ be an element of $\ZZZ^{V^-}$ with 
$$
\mu_1(z)\define
\left\{
\begin{array}{ll}
1&\quad z\in V,\\
N&\quad z=-\infty.
\end{array}
\right.
$$
Then it is easily seen that $\mu_1\in\UUUUU^{(1)}$.
Thus, $T^{\mu_1}\in\omega$ by Corollary \ref{cor:in e and omega}.
Since $T^\mu\in\omega^{(n)}=\ekhsg:\omega^m$,
we see that
$T^{\mu+m\mu_1}=T^\mu(T^{\mu_1})^m \in\ekhsg$.
Thus, by Corollary \ref{cor:in e and omega}, we see that
$(\mu+m\mu_1)(z)\geq 0$ and therefore
$$\mu(z)\geq -m=n \mbox{ for any $z\in V$}.
$$

Next, let $K$ be an arbitrary maximal clique of $G$.
Set
$$
\mu_2(z)\define
\left\{
\begin{array}{ll}
N&\quad z\in K,\\
1&\quad z\in V\setminus K,\\
(\#K)N+1&\quad z=-\infty.
\end{array}
\right.
$$
We will show that $\mu_2\in\UUUUU^{(1)}$.
First $\mu_2(z)\geq 1$ for any $z\in V$ by the definition of $\mu_2$.
Moreover, let $K'$ be an arbitrary maximal clique of $G$.
If $K'=K$, then
$\mu_2^+(K')=\mu_2^+(K)=(\#K)N= \mu_2(-\infty)-1$.
If $K'\neq K$, then since $\#(K'\cap K)\leq \#K-1$,
$\mu^+_2(K')<(\#K-1)N+\#V<(\#K)N=\mu_2(-\infty)-1$.
Further, let $C$ be an arbitrary odd cycle without chord and length at least 5.
If $K\cap C=\emptyset$, then 
$\mu_2^+(C)\leq \#V\leq N\frac{\#C-1}{2}-1\leq\mu_2(-\infty)\frac{\#C-1}{2}-1$.
Suppose that $K\cap C\neq \emptyset$.
Then $\#(K\cap C)\leq 2$, since $C$ does not have a chord,
and $\#K\geq 2$, since $K$ is a maximal clique.
Thus,
\begin{eqnarray*}
\mu_2^+(C)&\leq&2N+\#V\\
&\leq&3N-1\\
&\leq&2N\frac{\#C-1}{2}-1\\
&\leq&\mu_2(-\infty)\frac{\#C-1}{2}-1.
\end{eqnarray*}
Therefore, $\mu_2\in\UUUUU^{(1)}$ and we see by Corollary \ref{cor:in e and omega} that
$T^{\mu_2}\in\omega$.
Since $T^\mu\in\omega^{(n)}=\ekhsg:\omega^m$,
we see that
$T^{\mu+m\mu_2}=T^\mu(T^{\mu_2})^m\in\ekhsg$
and therefore, $\mu+m\mu_2\in \UUUUU^{(0)}$.
Thus,
$(\mu+m\mu_2)^+(K)\leq (\mu+m\mu_2)(-\infty)$.
Since $\mu_2^+(K)=(\#K)N$ and $\mu_2(-\infty)=(\#K)N+1$, we see that
$$
\mu^+(K)\leq\mu(-\infty)+m=\mu(-\infty)-n.
$$

Next, let $C$ be an arbitrary odd cycle without chord and length at least 5.
Take $c_0\in C$ and set
$$
\mu_3(z)\define
\left\{
\begin{array}{ll}
N(\#C-1)&\quad z\in C\setminus\{c_0\},\\
N(\#C-1)-1&\quad z=c_0,\\
1&\quad z\in V\setminus C,\\
2N(\#C)&\quad z=-\infty.
\end{array}
\right.
$$
We will show that $\mu_3\in\UUUUU^{(1)}$.
It is clear that $\mu_3(z)\geq 1$ for any $z\in V$.
Let $K$ be an arbitrary maximal clique of $G$.
Then $\#(K\cap C)\leq 2$ and therefore
$$
\mu_3^+(K)\leq 2N(\#C-1)+\#V
\leq 2N(\#C)-1
=\mu_3(-\infty)-1.
$$
Let $C'$ be an arbitrary odd cycle without chord and length at least 5.
If $C'=C$, then
\begin{eqnarray}
\mu_3^+(C')&=&\mu_3^+(C)\nonumber\\
&=&(\#C)N(\#C-1)-1\nonumber\\
&=&2N(\#C)\frac{\#C-1}{2}-1\nonumber\\
&=&\mu_3(-\infty)\frac{\#C'-1}{2}-1.
\mylabel{eq:mu3c}
\end{eqnarray}
Suppose that $C'\neq C$.
Then $\#(C\cap C')\leq \#C'-1$, since $C$ does not have a chord.
Therefore,
\begin{eqnarray*}
\mu_3^+(C')&\leq&(\#C'-1)N(\#C-1)+\#V\\
&\leq& N(\#C)(\#C'-1)-1\\
&=& 2N(\#C)\frac{\#C'-1}{2}-1\\
&=&\mu_3(-\infty)\frac{\#C'-1}{2}-1.
\end{eqnarray*}
Thus, we see that $\mu_3\in\UUUUU^{(1)}$ and by Corollary \ref{cor:in e and omega} that
$T^{\mu_3}\in\omega$.
Since $T^\mu\in\omega^{(n)}=\ekhsg:\omega^m$, we see that
$T^{\mu+m\mu_3}=T^\mu(T^{\mu_3})^m\in\ekhsg$
and therefore $\mu+m\mu_3\in\UUUUU^{(0)}$.
Thus,
$$
(\mu+m\mu_3)^+(C)\leq(\mu+m\mu_3)(-\infty)\frac{\#C-1}{2}.
$$
Since
$\mu_3^+(C)=\mu_3(-\infty)\frac{\#C-1}{2}-1$ by equation \refeq{eq:mu3c}, we see that
$$
\mu^+(C)\leq\mu(-\infty)\frac{\#C-1}{2}+m=\mu(-\infty)\frac{\#C-1}{2}-n.
$$
Thus, we see that $\mu\in\UUUUU^{(n)}$.

Next we consider the case where $n\geq 2$.
If $\mu\in\UUUUU^{(n)}$, then for any $\mu'\in\UUUUU^{(-n)}$,
$\mu+\mu'\in\UUUUU^{(0)}$ by the fact noted after Definition \ref{def:un}.
Since $\omega^{(-n)}$ is the $\KKK$-vector space with basis $T^{\mu'}$ with
$\mu'\in\UUUUU^{(-n)}$ by the case shown above, we see that
$T^\mu\in\ekhsg:\omega^{(-n)}=\omega^{(n)}$.

Conversely, assume that $T^\mu\in\omega^{(n)}$.
%Let $N$ be a huge integer ($N>\#V$).
Set
$$
\mu_4(z)\define
\left\{
\begin{array}{ll}
-1&\quad z\in V,\\
0&\quad z=-\infty.
\end{array}
\right.
$$
Then
%, since $N$ is a huge integer, 
it is easily verified that
$\mu_4\in\UUUUU^{(-1)}$ and therefore
by the fact shown above, we see that $T^{\mu_4}\in\omega^{(-1)}$.
Therefore, $T^{\mu+n\mu_4}=T^\mu(T^{\mu_4})^n \in\ekhsg$
and we see by Corollary \ref{cor:in e and omega} that
$\mu+n\mu_4\in\UUUUU^{(0)}$.
Thus, $(\mu+n\mu_4)(z)\geq 0$ for any $z\in V$ and we see that 
$$\mu(z)\geq n$$
for any $z\in V$.
Next, let $K$ be an arbitrary maximal clique of $G$.
Take $k_0\in K$ and  set
$$
\mu_5(z)\define
\left\{
\begin{array}{ll}
0&\quad z\in (K\setminus\{k_0\})\cup\{-\infty\},\\
1&\quad z=k_0,\\
-1&\quad z\in V\setminus K.
\end{array}
\right.
$$
Then, 
it is easily verified that
%we see that 
$\mu_5\in \UUUUU^{(-1)}$
and $\mu_5^+(K)=1$.
In particular, $T^{\mu_5}\in\omega^{(-1)}$.
Since $\mu+n\mu_5\in\UUUUU^{(0)}$ by the same reason as above, we see that
$$
(\mu+n\mu_5)^+(K)\leq \mu(-\infty)+n\mu_5(-\infty)=\mu(-\infty).
$$
Therefore,
$$
\mu^+(K)\leq\mu(-\infty)-n,
$$
since $\mu_5^+(K)=1$.
Finally, let $C$ be an arbitrary odd cycle without chord and length at least 5.
Take $c_0\in C$ and set
$$
\mu_6(z)\define
\left\{
\begin{array}{ll}
0&\quad z\in (C\setminus\{c_0\})\cup\{-\infty\},\\
1&\quad z=c_0,\\
-1&\quad z\in V\setminus C.
\end{array}
\right.
$$
Then, 
it is easily verified that
%we see that 
$\mu_6\in\UUUUU^{(-1)}$ and 
$
\mu_6^+(C)=1
$.
In particular, $T^{\mu_6}\in\omega^{(-1)}$.
%by a similar argument for the case of $\mu_3$.
Thus, $\mu+n\mu_6\in\UUUUU^{(0)}$ and
since
$$
(\mu+n\mu_6)^+(C)\leq (\mu+n\mu_6)(-\infty)\frac{\#C-1}{2}
=\mu(-\infty)\frac{\#C-1}{2},
$$
we see that 
$$
\mu^+(C)\leq\mu(-\infty)\frac{\#C-1}{2}-n,
$$
since $\mu_6^+(C)=1$.
Therefore, $\mu\in\UUUUU^{(n)}$.
\end{proof}

Now we show the following.

\begin{thm}
\mylabel{thm:gor char}
$\ekhsg$ is \gor\ if and only if
\begin{enumerate}
\item
\mylabel{item:clique size}
Sizes of maximal cliques are constant (say $n$) and
\item
\mylabel{item:3cond}
\begin{enumerate}
\item
\mylabel{item:1}
$n=1$, 
\item
\mylabel{item:2}
$n=2$ and there is no odd cycle without chord and length at least 7 or
\item
\mylabel{item:3}
$n\geq 3$ and there is no odd cycle without chord and length at least 5.
\end{enumerate}
\end{enumerate}
In particular, if $G$ is an h-perfect graph, 
then $\eksg$ is \gor\ if and only if \ref{item:clique size} and \ref{item:3cond}
above are satisfied.
\end{thm}
\begin{proof}
We denote $\omega_\ekhsg$ by $\omega$.
We first prove the ``if'' part.
Set $\eta(v)=1$, $\zeta(v)=-1$ for $v\in V$,
$\eta(-\infty)=n+1$ and $\zeta(-\infty)=-n-1$.
We will show that $\eta\in\UUUUU^{(1)}$ and $\zeta\in\UUUUU^{(-1)}$.

First, it is clear from the definition that $\eta(v)\geq 1$ and $\zeta(v)\geq -1$ for any
$v\in V$ and
$\eta^+(K)\leq \eta(-\infty)-1$ and $\zeta^+(K)\leq\zeta(-\infty)+1$
for any maximal clique $K$ of $G$.
If \ref{item:1} is satisfied, then $E=\emptyset$ and there is no odd cycle.
Thus, we see that $\eta\in\UUUUU^{(1)}$ and $\zeta\in\UUUUU^{(-1)}$
by the trivial reason.
We also see that  
if \ref{item:3} is satisfied, then
$\eta\in\UUUUU^{(1)}$ and $\zeta\in\UUUUU^{(-1)}$ 
by the trivial reason.
Now suppose that \ref{item:2} is satisfied.
If $C$ is an odd cycle without chord and length at least 5, then the length of $C$
is 5 by the assumption.
Since $\eta(v)=1$ (resp.\ $\zeta(v)=-1$) for any $v\in V$ and $\eta(-\infty)=3$
(resp.\ $\zeta(-\infty)=-3$), we see that
$$
\eta^+(C)=5=3\cdot\frac{5-1}{2}-1=\eta(-\infty)\frac{\#C-1}{2}-1
$$
(resp.\ $\zeta^+(C)=-5=-3\cdot\frac{5-1}{2}+1=\eta(-\infty)\frac{\#C-1}{2}+1$).
Therefore, $\eta\in\UUUUU^{(1)}$ and $\zeta\in\UUUUU^{(-1)}$.

Since $\eta+\zeta=0$, we see by Proposition \ref{prop:symb power} and Fact \ref{fact:hhs1.1}
that
$$
1=T^0\in\omega^{-1}\omega=\trace(\omega).
$$
Thus, we see by Fact \ref{fact:hhs2.1} that $\ekhsg$ is \gor.

Now we prove the ``only if'' part.
By Facts \ref{fact:hhs1.1} and \ref{fact:hhs2.1}, we see that
$$
1\in\omega^{-1}\omega=\omega^{(-1)}\omega.
$$
Therefore, by Proposition \ref{prop:symb power}, we see that there are
$\eta\in\UUUUU^{(1)}$ and $\zeta\in\UUUUU^{(-1)}$ with $\eta+\zeta=0$.
Let $v$ be an arbitrary element of $V$.
Then, since $\eta(v)\geq 1$ and $\zeta(v)\geq -1$ and $(\eta+\zeta)(v)=0$,
we see that $\eta(v)=1$ and $\zeta(v)=-1$.
Let $K$ and $K'$ be arbitrary maximal cliques of $G$.
If $\#K>\#K'$, then
\begin{eqnarray*}
\eta(-\infty)&\geq&\eta^+(K)+1=\#K+1,\\
\zeta(-\infty)&\geq&\zeta^+(K')-1=-(\#K')-1\\
\end{eqnarray*}
and therefore,
$$
(\eta+\zeta)(-\infty)\geq\#K-(\#K')>0.
$$
This contradicts to $\eta+\zeta=0$.
Therefore $\#K\leq \#K'$ and by symmetry, we see that $\#K=\#K'$.
Since $K$ and $K'$ are arbitrary maximal cliques of $G$, we see that
the size of maximal cliques of $G$ are constant.

Let $n$ be this constant size.
Then since 
$$
\eta(-\infty)\geq n+1,
\quad
\zeta(-\infty)\geq -n-1
\quad\mbox{and}\quad
(\eta+\zeta)(-\infty)=0,
$$
we see that
$$
\eta(-\infty)=n+1
\quad\mbox{and}\quad
\zeta(-\infty)=-n-1.
$$

Suppose that there is an odd cycle $C$ without chord and length at least 5.
Then $n\geq 2$,
$$
\eta^+(C)\leq\eta(-\infty)\frac{\#C-1}{2}-1
\quad\mbox{and}\quad
\zeta^+(C)\leq\zeta(-\infty)\frac{\#C-1}{2}+1.
$$
Thus, since $\eta^+(C)=\#C$, $\zeta^+(C)=-(\#C)$,
$\eta(-\infty)=n+1$ and $\zeta(-\infty)=-n-1$,
we see that
$$
\frac{2(\#C+1)}{\#C-1}=\frac{2(\eta^+(C)+1)}{\#C-1}\leq\eta(-\infty)=n+1
$$
and
$$
-\frac{2(\#C+1)}{\#C-1}=\frac{2(\zeta^+(C)-1)}{\#C-1}\leq\zeta(-\infty)=-n-1,
$$
i.e.,
$$
n+1\leq 2+\frac{4}{\#C-1}\leq n+1.
$$
Therefore, $\#C=5$ and $n=2$, since $n$ is an integer.

Thus, we see that if $n\geq 3$, then there is no odd cycle without chord and length at least 5.
Moreover, we see that if $n=2$, then there is no odd cycle without chord and length at least 7.
\end{proof}

Set
$$
\tstab(G)\define
\left\{f\in\RRR^V \left|\ 
\vcenter{\hsize=.5\textwidth\relax\noindent
$0\leq f(x)\leq 1$ for any $x\in V$,
$f^+(e)\leq 1$ for any $e\in E$ and
$f^+(C)\leq\frac{\#C-1}{2}$ for any odd cycle $C$}
\right.\right\}
$$
and
$$
\qstab(G)\define
\left\{f\in\RRR^V \left|
\vcenter{\hsize=.5\textwidth\relax\noindent
$f(x)\geq 0$ for any $x\in V$ and 
$f^+(K)\leq 1$ for any clique $K$ in $G$}
\right.\right\}.
$$
If $\stab(G)=\tstab(G)$, then $G$ is called a t-perfect graph.
It is easily verified that $\hstab(G)\subset\tstab(G)$.
In particular, t-perfect graphs are h-perfect.
Further, by \cite[Theorem 3.1]{chv}, $G$ is perfect if and only if
$\stab(G)=\qstab(G)$.
Since $\hstab(G)\subset\qstab(G)$ by the definition,
perfect graphs are h-perfect.

\begin{cor}
\begin{enumerate}
\item
\mylabel{item:t-perfect}
Suppose that $G$ is t-perfect.
Then $\eksg$ is \gor\ if and only if
\begin{enumerate}
\item
%\mylabel{item:clique size}
Sizes of maximal cliques are constant (say $n$) and
\item
%\mylabel{item:3cond}
\begin{enumerate}
\item
$n=1$, 
\item
$n=2$ and there is no odd cycle without chord and length at least 7 or
\item
$n=3$ and there is no odd cycle without chord and length at least 5.
\end{enumerate}
\end{enumerate}
\item
\mylabel{item:perfect}
(Ohsugi-Hibi)
Suppose that $G$ is perfect.
Then $\eksg$ is \gor\ if and only if
sizes of maximal cliques are constant.
\end{enumerate}
\end{cor}
\begin{proof}
\ref{item:t-perfect}:
Since a t-perfect graph is h-perfect and 
has no clique with size more than 3, the result follows from Theorem \ref{thm:gor char}.

\ref{item:perfect}:
Since a perfect graph 
is h-perfect and 
has no odd cycle without chord and length at least 5,
the result follows from Theorem \ref{thm:gor char}.
\end{proof}

\begin{remark}
\rm
Set
%%%%%%%%%%%%%%%%%%%%%%%%%%
%%%%%%%%%%%%%%%%%%%%%%%%%%
%%%%%%%%%%%%%%%%%%%%%%%%%%
\iffalse
$$
\msKKK\define
\left\{K\subset V\left|\ 
\vcenter{\hsize=.5\textwidth\relax\noindent
$K$ is a clique of $G$ and size of $K$ is less than or equal to 3}
\right.\right\}.
$$
\fi
%%%%%%%%%%%%%%%%%%%%%%%%
%%%%%%%%%%%%%%%%%%%%%%%%
%%%%%%%%%%%%%%%%%%%%%%%%
$
\msKKK\define
\{K\subset V\mid
K$ is a clique of $G$ and size of $K$ is less than or equal to 3$\}$.
Then
$$
\tstab(G)=
\left\{f\in\RRR^V\left|\ 
\vcenter{\hsize=.5\textwidth\relax\noindent
$f(x)\geq 0$ for any $x\in V$,
$f^+(K)\leq 1$ for any maximal element $K$ of $\msKKK$ and
$f^+(C)\leq\frac{\#C-1}{2}$ for any odd cycle $C$ without chord
and length at least 5}
\right.\right\}.
$$
By defining 
$\tUUUUU^{(n)}\define\{\mu\in\ZZZ^{V^-}\mid
\mu(z)\geq n$ for any $z\in V$,
$\mu^+(K)\leq \mu(-\infty)-n$ for any maximal element $K$ of $\msKKK$ and
$\mu^+(C)\leq\mu(-\infty)\frac{\#C-1}{2}-n$ for any odd cycle $C$ without
chord and length at least $5\}$
for $n\in \ZZZ$,
one can verify that
$T^\mu\in\omega_{\ektsg}^{(n)}$ if and only if 
$\mu\in\tUUUUU^{(n)}$
along the same line as the proof of Proposition \ref{prop:symb power}.
Further, it is verified along the same line as the proof of Theorem \ref{thm:gor char}
that $\ektsg$ is \gor\ if and only if
\begin{enumerate}
\item
Sizes of maximal elements of $\msKKK$ are constant (say $n$) and
\item
\begin{enumerate}
\item
$n=1$, 
\item
$n=2$ and there is no odd cycle without chord and length at least 7 or
\item
$n=3$ and there is no odd cycle without chord and length at least 5,
\end{enumerate}
\end{enumerate}
i.e.,
\begin{enumerate}
\item
$E=\emptyset$,
\item
$G$ has no isolated vertex nor triangle and there is no odd cycle
without chord and length at least 7 or
\item
all maximal cliques of $G$ have size at least 3
and there is no odd cycle without chord and length at least 5.
\end{enumerate}

Similarly, by defining
$\qUUUUU^{(n)}\define\{\mu\in\ZZZ^{V^-}\mid
\mu(z)\geq n$ for any $z\in V$ and
$\mu^+(K)\leq \mu(-\infty)-n$ for any maximal clique of $G\}$
for $n\in \ZZZ$,
one can verify that
$T^\mu\in\omega_{\ekqsg}^{(n)}$ if and only if 
$\mu\in\qUUUUU^{(n)}$
along the same line as the proof of Proposition \ref{prop:symb power}.
Further, it is verified along the same line as the proof of Theorem \ref{thm:gor char}
that $\ekqsg$ is \gor\ if and only if
sizes of maximal cliques are constant.
\end{remark}

A graph $G=(V,E)$ is called a comparability graph if $V$ is a 
partially ordered set (poset for short) and 
$E=\{\{a,b\}\mid a<b$ or $b<a$ in $V\}$.
If $G$ is a comparability graph, then $G$ is perfect and 
$\stab(G)$, $\qstab(G)$ and $\hstab(G)$ coincide with the
chain polytope of $V$ (cf. \cite{sta3}).
We showed that the symbolic powers of the canonical ideal of the
Ehrhart ring of the chain polytope of a poset coincide with the
ordinary powers \cite[Theorem 3.8]{mc}.

Unlike the case of the chain polytope of a poset,
the symbolic power of the canonical ideal 
of the Ehrhart ring of $\hstab(G)$ 
is different from
the ordinary power in general, as the following example shows.

\begin{example}
\rm
Let $X\define\{x_i\mid 0\leq i\leq 6\}$,
$Y\define\{y_i\mid 0\leq i\leq 6\}$,
$Z\define\{z_i\mid 0\leq i\leq 6\}$,
$V\define X\cup Y\cup Z$ and
$$
E\define\{\{x_i,x_{i+1}\}, \{x_i,y_i\}, \{y_i,z_i\}, \{z_i, y_{i+1}\},
\{z_i, z_{i+3}\}\mid 0\leq i\leq 6\},
$$
where indices are considered 
modulo 7,
and set $G\define(V,E)$.
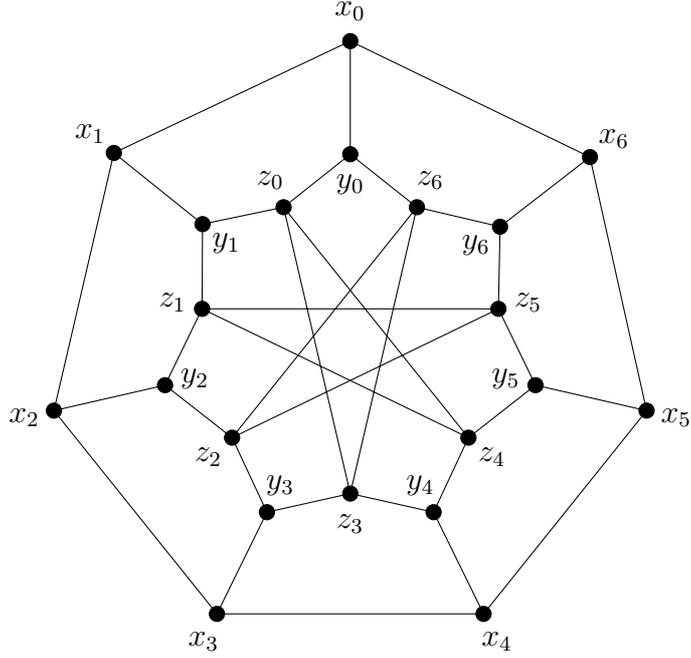
\begin{center}
\begin{tikzpicture}
\coordinate (X0) at (90:4);
\coordinate (X1) at (141:4);
\coordinate (X2) at (193:4);
\coordinate (X3) at (244:4);
\coordinate (X4) at (296:4);
\coordinate (X5) at (347:4);
\coordinate (X6) at (398:4);

\node at (90:4.4) {$x_0$};
\node at (141:4.4) {$x_1$};
\node at (193:4.4) {$x_2$};
\node at (244:4.4) {$x_3$};
\node at (296:4.4) {$x_4$};
\node at (347:4.4) {$x_5$};
\node at (398:4.4) {$x_6$};

\coordinate (Y0) at (90:2.5);
\coordinate (Y1) at (141:2.5);
\coordinate (Y2) at (193:2.5);
\coordinate (Y3) at (244:2.5);
\coordinate (Y4) at (296:2.5);
\coordinate (Y5) at (347:2.5);
\coordinate (Y6) at (398:2.5);

\node at (90:2.1) {$y_0$};
\node at (141:2.1) {$y_1$};
\node at (193:2.1) {$y_2$};
\node at (244:2.1) {$y_3$};
\node at (296:2.1) {$y_4$};
\node at (347:2.1) {$y_5$};
\node at (398:2.1) {$y_6$};

\coordinate (Z0) at (116:2);
\coordinate (Z1) at (167:2);
\coordinate (Z2) at (219:2);
\coordinate (Z3) at (270:2);
\coordinate (Z4) at (321:2);
\coordinate (Z5) at (373:2);
\coordinate (Z6) at (424:2);

\node at (116:2.4) {$z_0$};
\node at (167:2.4) {$z_1$};
\node at (219:2.4) {$z_2$};
\node at (270:2.4) {$z_3$};
\node at (321:2.4) {$z_4$};
\node at (373:2.4) {$z_5$};
\node at (424:2.4) {$z_6$};

\draw (X0)--(X1)--(X2)--(X3)--(X4)--(X5)--(X6)--cycle;
\draw (X0)--(Y0);
\draw (X1)--(Y1);
\draw (X2)--(Y2);
\draw (X3)--(Y3);
\draw (X4)--(Y4);
\draw (X5)--(Y5);
\draw (X6)--(Y6);
\draw (Y0)--(Z0)--(Y1)--(Z1)--(Y2)--(Z2)--(Y3)--(Z3)--(Y4)--(Z4)--(Y5)--(Z5)--(Y6)--(Z6)--cycle;
\draw (Z0)--(Z3)--(Z6)--(Z2)--(Z5)--(Z1)--(Z4)--cycle;

\draw[fill] (X0) circle [radius=0.1];
\draw[fill] (X1) circle [radius=0.1];
\draw[fill] (X2) circle [radius=0.1];
\draw[fill] (X3) circle [radius=0.1];
\draw[fill] (X4) circle [radius=0.1];
\draw[fill] (X5) circle [radius=0.1];
\draw[fill] (X6) circle [radius=0.1];
\draw[fill] (Y0) circle [radius=0.1];
\draw[fill] (Y1) circle [radius=0.1];
\draw[fill] (Y2) circle [radius=0.1];
\draw[fill] (Y3) circle [radius=0.1];
\draw[fill] (Y4) circle [radius=0.1];
\draw[fill] (Y5) circle [radius=0.1];
\draw[fill] (Y6) circle [radius=0.1];
\draw[fill] (Z0) circle [radius=0.1];
\draw[fill] (Z1) circle [radius=0.1];
\draw[fill] (Z2) circle [radius=0.1];
\draw[fill] (Z3) circle [radius=0.1];
\draw[fill] (Z4) circle [radius=0.1];
\draw[fill] (Z5) circle [radius=0.1];
\draw[fill] (Z6) circle [radius=0.1];

\end{tikzpicture}
\end{center}

Note that $Y$ is a stable set, 
$X$ and $Z$ are 7-cycles without chord and no elements of $X$ and $Z$
are adjacent.

We first show that $G$ has no triangle.
Assume the contrary and let $T$ be a triangle in $G$.
Since no elements of $X$ and $Z$ are adjacent, we see that
$T\subset X\cup Y$ or $T\subset Z\cup Y$.
Since $X$ and $Z$ are 7-cycles without chord, $T$ can not be contained 
in $X$ or $Z$.
Further, since $Y$ is a stable set, $T$ can not contain 2 or more elements
of $Y$.
Thus, $T$ contains exactly 1 element of $Y$.
By symmetry, we may assume that $T\cap Y=\{y_0\}$.
If $T\subset X\cup Y$, then $T$ contains 2 elements of $X$ which are 
adjacent to $y_0$.
However, there is only 1 element of $X$ which is adjacent to $y_0$.
Thus, $T\subset Z\cup Y$.
Since elements of $Z$ adjacent to $y_0$ are $z_0$ and $z_6$,
we see that $T=\{y_0, z_0, z_6\}$.
This is a contradiction since $z_0$ and $z_6$ are not adjacent.

Next we show that for any 5-cycle $C$ in $G$,
$C\cap X\neq\emptyset$ and $C\cap Z\neq\emptyset$.
In particular, $\#(C\cap Y)\geq 2$, since no elements of $X$ and
$Z$ are adjacent.
Assume the contrary.
Then $C\subset X\cup Y$ or $C\subset Z\cup Y$.
Since $X$ and $Z$ are 7-cycles without chord, we see that 
$C\not\subset X$ and $C\not \subset Z$.
Thus, $C\cap Y\neq \emptyset$.
By symmetry, we may assume that $C\ni y_0$.
Since $y_0$ is adjacent to only 1 element $x_0$ in $X\cup Y$,
$C\subset X\cup Y$ is impossible.
Thus, $C\subset Z\cup Y$.
$y_0$ is adjacent to exactly 2 elements $z_0$ and $z_6$ in $Z\cup Y$.
Further, $z_0$ is adjacent to $z_3$, $z_4$ and $y_1$ and
$z_6$ is adjacent to $z_2$, $z_3$ and $y_6$ in $Z\cup Y$
except $y_0$.
Since no elements of $\{z_3, z_4, y_1\}$ and $\{z_2, z_3, y_6\}$ are
adjacent, we see that this case is also impossible.

Define $\mu\in\ZZZ^{V^-}$ by $\mu(x_i)=\mu(z_i)=4$, $\mu(y_i)=3$ for $0\leq i\leq 6$ and
$\mu(-\infty)=10$.
Since $\max_{w\in V}\mu(w)=4$, we see that 
$\mu^+(e)\leq 8=\mu(-\infty)-2$ for any $e\in E$.
Let $C$ be an arbitrary 5-cycle.
Since $\#(C\cap Y)\geq 2$, we see that
$$
\mu^+(C)\leq 18=\mu(-\infty)\frac{5-1}{2}-2.
$$
Let $C$ be an arbitrary odd cycle with length at least 7.
Then, since $\max_{w\in V}\mu(w)=4$,
we see that 
$$
\mu^+(C)\leq 4(\#C)\leq 5(\#C)-7=\mu(-\infty)\frac{\#C-1}{2}-2.
$$
Thus, we see that $\mu\in\UUUUU^{(2)}$ and therefore $T^\mu\in\omega_{\ekhsg}^{(2)}$
by Proposition \ref{prop:symb power}.

We deduce a contradiction by assuming that there are $\mu_1$, $\mu_2\in\UUUUU^{(1)}$ with
$\mu=\mu_1+\mu_2$.
We may assume that $\mu_1(-\infty)\leq\mu_2(-\infty)$.

Since $X$ is a 7-cycle without chord, we see that
$$
\mu_i^+(X)\leq\mu_i(-\infty)\frac{7-1}{2}-1=3\mu_i(-\infty)-1
$$
for $i=1$, $2$.
Further, since
$$
\mu_1^+(X)+\mu_2^+(X)=\mu^+(X),\qquad
\mu_1(-\infty)+\mu_2(-\infty)=\mu(-\infty)
$$
and
$$
\mu^+(X)=28=3\cdot 10-2=3\mu(-\infty)-2,
$$
we see that
$$
\mu_i^+(X)=3\mu_i(-\infty)-1
$$
for $i=1$, $2$.
We see that
$$
\mu_i^+(Z)=3\mu_i(-\infty)-1
$$
for $i=1$, $2$ by the same way.

Suppose that $\mu_1(-\infty)\leq 2$.
Then $\mu_1^+(e)\geq 2\geq \mu_1(-\infty)$ for any $e\in E$,
which contradicts to the clique inequality $\mu_1^+(e)\leq \mu_1(-\infty)-1$.
Thus $\mu_1(-\infty)\geq 3$.

Suppose that $\mu_1(-\infty)=3$.
Then since $2\leq \mu_1^+(e)\leq\mu_1(-\infty)-1$ for any $e\in E$ and
$\mu_1(v)\geq 1$ for any $v\in V$, we see that $\mu_1(v)=1$ for any $v\in V$.
Thus,
$$
\mu_1^+(X)=7<8=3\mu_1(-\infty)-1.
$$
This contradicts to the fact shown above.

Next assume that $\mu_1(-\infty)=4$.
Then for any $e\in E$, $\mu_1^+(e)\leq \mu_1(-\infty)-1=3$ and therefore
$2\mu_1^+(X)=\sum_{i=0}^6(\mu_1(x_i)+\mu_1(x_{i+1}))\leq 21$.
Thus, $\mu_1^+(X)\leq 10$, since $\mu_1^+(X)$ is an integer and
$$
\mu_1^+(X)\leq 10<11=3\mu_1(-\infty)-1.
$$
This is a contradiction again.

Finally, suppose that $\mu_1(-\infty)=5$.
Then for any $e\in E$, $\mu_1^+(e)\leq 4=\mu_1(-\infty)-1$.
Since
$$
2\mu_1^+(X)=\sum_{j=0}^6(\mu_1(x_j)+\mu_1(x_{j+1}))\leq 7\cdot 4=28
$$
and
$$
\mu_1^+(X)=3\mu_1(-\infty)-1=14,
$$
we see that $\mu_1(x_j)+\mu_1(x_{j+1})=4$ for $0\leq j\leq 6$.
By solving system of these linear equations, we see that 
$\mu_1(x_j)=2$ for $0\leq j\leq 6$.
We also see that $\mu_1(z_j)=2$ for $0\leq j\leq 6$ by
the same way.

Consider the 5-cycle $C_j=x_jy_jz_jy_{j+1}x_{j+1}x_j$ for $0\leq j\leq 6$.
Since 
$$
\mu^+(C_j)=18=\mu(-\infty)\frac{\#C_j-1}{2}-2,
$$
we see by the same argument as above that
$\mu_i^+(C_j)=9$ for $i=1$, $2$ and $0\leq j\leq 6$.
Since $\mu_1(x_j)=\mu_1(x_{j+1})=\mu_1(z_j)=2$,
$\mu_1(y_k)\geq 1$ for $k=j$, $j+1$,
we see that
$$\mu_1(y_j)=1\quad\mbox{and}\quad \mu_1(y_{j+1})=2
$$
or
$$\mu_1(y_j)=2\quad\mbox{and}\quad \mu_1(y_{j+1})=1.
$$
Suppose that $\mu_1(y_0)=1$.
Then $\mu_1(y_1)=2$, $\mu_1(y_2)=1$, \ldots, $\mu_1(y_6)=1$ and $\mu_1(y_0)=2$.
This is a contradiction.
We also deduce a contradiction if $\mu_1(y_0)=2$.

Thus, we see that there are no $\mu_1$, $\mu_2\in\UUUUU^{(1)}$ with
$\mu=\mu_1+\mu_2$.
Therefore $T^\mu\not\in\omega_{\ekhsg}^2$ and we see that
$\omega_{\ekhsg}^2\subsetneq\omega_{\ekhsg}^{(2)}$.

The graph $G$ in this example is not h-perfect.
In fact, let $\nu$ be an element of $\ZZZ^V$ with
$\nu(x_i)=\nu(z_i)=3/7$ and $\nu(y_i)=5/14$ for $0\leq i\leq 6$.

We first show that $\nu\in\hstab(G)$.
It is clear that $\nu(v)\geq 0$ for any $v\in V$.
Since $\max_{v\in V}\nu(v)=3/7$, we see that $\nu^+(e)\leq 1$ 
for any $e\in E$.
Since $G$ has no triangle, we see that $\nu^+(K)\leq 1$ for any clique
$K$ in $G$.
Let $C$ be an arbitrary odd cycle without chord and length at least 5.
If the length of $C$ is greater than or equals to 7, then
$$
\nu^+(C)\leq \frac{3}{7}(\#C)=\frac{1}{2}(\#C)-\frac{1}{14}(\#C)
\leq\frac{1}{2}(\#C)-\frac{1}{14}\cdot 7
=\frac{\#C-1}{2}.
$$
If $\#C=5$, then $C$ is one of $C_i$ above.
Therefore,
$$
\nu^+(C)=\nu^+(C_i)=3\cdot\frac{3}{7}+2\cdot\frac{5}{14}=2=\frac{\#C-1}{2}.
$$
Thus, we see that $\nu\in\hstab(G)$.

Next we show that $\nu\not\in\stab(G)$.
Assume the contrary.
Then there are stable sets $S_1$, \ldots, $S_m$ of $G$ and
positive real numbers $\lambda_1$, \ldots, $\lambda_m$ such that
$\lambda_1+\cdots+\lambda_m=1$ and
$$
\nu=\sum_{j=1}^m\lambda_j\chi_{S_j}.
$$
Since $X$ is a 7-cycle, we see that
$
\chi^+_{S_j}(X)=\#(S_j\cap X)\leq3
$
for any $0\leq j\leq m$.
On the other hand, since
$$
3=\nu^+(X)=\sum_{j=1}^m\lambda_j\chi^+_{S_j}(X)
\qquad
\mbox{and}
\qquad
\lambda_1+\cdots+\lambda_m=1,
$$
we see that
$\chi^+_{S_j}(X)=3$
for any $0\leq j\leq m$.
We also see that
$\chi^+_{S_j}(Z)=3$ and $\chi^+_{S_j}(C_i)=2$ for
$0\leq i\leq 6$ and for $0\leq j\leq m$ by the same way.

Set $S\define S_1$.
Since $\#(S\cap X)=\chi^+_S(X)=3$ and $S$ is a stable set, we 
may assume, by symmetry, that $S\cap X=\{x_0, x_2, x_4\}$.
Consider $S\cap C_1$.
Since $x_2\in S$ and $S$ is a stable set,
we see that $y_2\not\in S$.
Thus,
$S\cap C_1=\{x_2, y_1\}$ or $S\cap C_1=\{x_2, z_1\}$.
We see that $S\cap C_2=\{x_2,y_3\}$ or $S\cap C_2=\{x_2, z_2\}$
by the same way.

First consider the case where $S\cap C_1=\{x_2, y_1\}$ and $S\cap C_2=\{x_2,y_3\}$.
Then $S\cap C_0=\{x_0, y_1\}$ and $S\cap C_3=\{x_4, y_3\}$ since
$\#(S\cap C_0)=\#(S\cap C_3)=2$.
Thus, $z_0$, $z_1$, $z_2$, $z_3\not\in S$.
Since $\#(S\cap Z)=3$, we see that $S\cap Z=\{z_4, z_5, z_6\}$.
Since $x_0$, $x_4\in S$ and $\#(S\cap C_4)=\#(S\cap C_6)=2$, we see that
$S\cap C_4=\{x_4, z_4\}$ and $S\cap C_6=\{x_0, z_6\}$.
Thus, we see that $x_5$, $y_5$, $x_6$, $y_6\not\in S$
and therefore, $\#(S\cap C_5)=1$.
This contradicts to the fact that $\#(S\cap C_5)=2$.

Next consider the case where $S\cap C_1=\{x_2, y_1\}$ and $S\cap C_2=\{x_2, z_2\}$.
Since $Z$ is a 7-cycle $z_0z_3z_6z_2z_5z_1z_4z_0$, $z_1\not\in S$, $z_2\in S$,
$\#(S\cap Z)=3$ and $S$ is a stable set, we see that $S\cap Z=\{z_2, z_3, z_4\}$.
Since $x_4$, $z_4\in S\cap C_4$ and $\#(S\cap C_4)=2$,
we see that $S\cap C_4=\{x_4, z_4\}$.
Therefore, $y_5\not\in S$.
Further, $x_5$, $x_6\not\in S$ by our assumption and $z_5\not\in S$ by the
fact shown above.
Therefore, $S\cap C_5\subset\{y_6\}$,
which contradicts to $\#(S\cap C_5)=2$.

The case where $S\cap C_1=\{x_2, z_1\}$ and $S\cap C_2=\{x_2, y_3\}$
can not occur by the same reason.

Finally, assume that $S\cap C_1=\{x_2, z_1\}$ and $S\cap C_2=\{x_2, z_2\}$.
Then $x_1$, $y_1\not\in S$ and since $y_0$ is adjacent to $x_0\in S$,
we see that $y_0\not\in S$.
Therefore, $S\cap C_0=\{x_0,z_0\}$ since $\#(S\cap C_0)=2$.
We see that $S\cap C_3=\{x_4, z_3\}$ by the same way.
This contradicts to the fact that $S$ is a stable set since $z_0$ and $z_3$
are adjacent.

Thus, we see that $\nu\not\in\stab(G)$ and therefore,
$\stab(G)\subsetneq\hstab(G)$.
\end{example}

%%%%%%%%%%%%%%%%%%%%%%%%%%%%%%%%%%%%%%%%%%%%%%

%===========================

%

\end{document}